\documentclass[12pt]{amsart}

\RequirePackage[colorlinks,citecolor=blue,urlcolor=blue]{hyperref}

\usepackage{amsmath,amstext,amssymb,amsopn,amsthm}
\usepackage{url,verbatim}
\usepackage{mathtools}
\usepackage{enumerate}

\usepackage{color,graphicx}
\usepackage{subfig}

\usepackage[margin=30mm]{geometry}
\usepackage{eucal,mathrsfs,dsfont}

\allowdisplaybreaks

\newtheorem{theorem}{Theorem}[section]
\newtheorem{corollary}[theorem]{Corollary}
\newtheorem{lemma}[theorem]{Lemma}

\theoremstyle{definition}

\newtheorem{definition}[theorem]{Definition}

\newtheorem{remark}[theorem]{Remark}

\numberwithin{equation}{section}

\newcommand{\eps}{\varepsilon}

\newcommand{\calF}{\mathcal{F}}

\newcommand{\calT}{\mathcal{T}}

\newcommand{\calW}{\mathcal{W}}

\newcommand{\calB}{\mathcal{B}}
\newcommand{\calV}{\mathcal{V}}
\newcommand{\calE}{\mathcal{E}}
\newcommand{\calX}{\mathcal{X}}

\newcommand{\R}{\mathds{R}}

\newcommand{\prt}{\partial}

\newcommand{\bF}{\mathbf{F}}

\newcommand{\wh}{\widehat}
\newcommand{\wt}{\widetilde}

\DeclareMathOperator{\dist}{dist}

\title[Billiard balls]{On pinned billiard balls and foldings}
\author{Jayadev S. Athreya, Krzysztof Burdzy and Mauricio Duarte}

\address{JSA and KB: Department of Mathematics, Box 354350, University of Washington, Seattle, WA 98195}
\email{jathreya@uw.edu}
\email{burdzy@uw.edu}

\address{MD: Departamento de Matematicas, Facultad de Ciencias Exactas, Universidad Andres Bello, Santiago, Chile}
\email{mauricio.duarte@unab.cl}

\thanks{JSA's research was supported in part by NSF CAREER grant DMS 1559860.}
\thanks{KB's research was supported in part by Simons Foundation Grant 506732. }
\thanks{MD's research was supported in part by Proyecto FONDECYT 11160591,  N\'ucleo Milenio NC130062, and Basal CONICYT Program PFB 03.}

\begin{document}

\begin{abstract}

We  consider systems of ``pinned balls,'' i.e.,  balls that  have fixed positions and pseudo-velocities. Pseudo-velocities change according to the same rules as those  for velocities of totally elastic collisions between moving balls. The times of collisions for different pairs of pinned balls are chosen in an exogenous way.
We  give an explicit upper bound for the maximum number of pseudo-collisions for
a system of $n$ pinned balls in a $d$-dimensional space, in terms of $n$, $d$ and the locations of ball centers.  
As a first step, we study foldings, i.e., mappings that formalize the idea of folding a piece of paper along a crease.

\end{abstract}

\maketitle

\section{Introduction}
\label{se:intro}

This paper is inspired by articles on the maximum number of totally elastic collisions for a finite system of balls in a billiard table with no walls (i.e., the whole Euclidean space). We will review the history of this problem in Section \ref{review}.

\subsection{Pinned Balls} The main concern of this article are systems of ``pinned balls.'' In this model, balls have positions and pseudo-velocities, and some balls are in contact with other balls. The balls do not move, i.e., their positions are constant as functions of time. But the pseudo-velocities change according to the same rules as those for velocities of totally elastic collisions between moving balls, except that the order in which these collisions occur must be prescribed, that is, the times of collisions for different pairs of touching balls are chosen in an exogenous way. The model is inspired by the usual system of moving and colliding balls. 

There are reasons to think that  a large number of collisions for a system of moving balls can occur only when many of the balls form a tight configuration. Theorem 1.3 in \cite{BuD18b} states that if a family of $n$ balls undergoes more than $n^{cn}$ collisions for an appropriate $c>0$, then there is a subfamily $\calB'$ of balls and an interval of time $[t_1,t_2]$ in which a large number of collisions occur between members of $\calB'$, and the balls in $\calB'$ form a very tight configuration during the whole interval (see  \cite[Thm.~1.3]{BuD18b} for the quantitative version of the qualitative statement presented here).

The main theorem in \cite{BD} (see Section \ref{review} below) is an example based in an essential way on the analysis of a pinned ball configuration.

\medskip

Let $\tau_d$ denote
the  kissing number  of a $d$-dimensional ball, i.e.,  the maximum
number of mutually nonoverlapping translates of the ball that can be arranged
so that they all touch the ball.
We will prove the following upper bound for the number of pseudo-collisions of pinned balls.

\begin{theorem}\label{au21.7}
A system $\bF$ of $n$ pinned balls in $d$-dimensional space
may have at most
$
\left(  2^{21/2} d n^5 /\alpha(\bF)\right)^{\tau_d n/2-1}$  collisions.

\end{theorem}

The number $\alpha(\bF)$ appearing in the theorem is an explicit, although complicated, function of the locations of ball centers. We will call $\alpha(\bF)$ the ``index of approximate rigidity.'' The definition of $\alpha(\bF)$ and the explanation of its name will be given in Section \ref{s28.1}.
See Section \ref{prelim} and Theorem \ref{au21.6} for the rigorous formulation of Theorem \ref{au21.7}.

\subsection{Foldings}
\label{sec:fold}
The article will start with some results on \emph{foldings} of Euclidean spaces $\R^d$.  A \emph{folding} with respect to a halfspace $H$  is the identity on $H$ and maps the complementary halfspace $H'$ onto $H$ via reflection in the hyperplane $P$ which is the common boundary $\partial H' = \partial H$. We will later see that collisions of pinned balls can be represented as foldings.
We will consider an infinite sequence of foldings corresponding to a finite number of halfspaces whose intersection has a non-empty interior.  Foldings with respect to the same hyperplane will appear repeatedly in the sequence.  We will prove in Theorem \ref{au20.1} that for every starting   point, its orbit generated by the sequence of foldings  is finite.

The theorem on foldings just mentioned
 implies that the number of collisions of pinned balls is finite for any positions of the balls, any initial velocities, and any order in which pairs of balls collide. However this approach does not yield any bound for the maximum number of collisions.

\subsection{Hard ball collisions---historical review}
\label{review}

The question of whether a finite system of hard balls can have an infinite number of elastic collisions was posed by Ya.~Sinai. It was answered in negative in \cite{Vaser79}. For alternative proofs see \cite{Illner89, Illner90,IllnerChen,BuD18b}. 
It was proved in \cite{BFK1} that a system of $n$ balls in the Euclidean space undergoing elastic collisions can experience at most
\begin{align}\label{s26.1}
\left( 32 \sqrt{\frac{m_{\text{max}}}{m_{\text{min}}} } 
\frac{r_{\text{max}}}{r_{\text{min}}} n^{3/2}\right)^{n^2}
\end{align}
collisions. Here $m_{\text{max}}$ and $m_{\text{min}}$ denote the maximum and the minimum masses of the balls. Likewise, $r_{\text{max}}$ and $r_{\text{min}}$ denote the maximum and the minimum radii of the balls.
The following alternative upper bound for the maximum number of collisions appeared in \cite{BFK5}
\begin{align}\label{s26.2}
\left( 400 \frac{m_{\text{max}}}{m_{\text{min}}} 
 n^2\right)^{2n^4}.
\end{align}
  The papers \cite{BFK1,BFK2, BFK3,BFK4, BFK5} were the first to present universal bounds \eqref{s26.1}-\eqref{s26.2} on the number of collisions of $n$ hard balls in any dimension. No improved universal bounds were found since then, as far as we know.

It has been proved in \cite{BD} by example that the  number of elastic collisions of $n$ balls
in $d$-dimensional space is greater than $n^3/27$ for $n\geq 3$ and $d\geq 2$, for some initial conditions. The previously known lower bound was of order $n^2$ (that bound was for balls in dimension 1 and was totally elementary).

\subsection{Organization of the paper}

Section \ref{fold} is devoted to foldings. Section \ref{prelim} contains the rigorous description of the pinned balls model, some elementary remarks and preliminary results. Section \ref{s28.1} presents the ``index of approximate rigidity'' and some estimates of this quantity. Geometry of ``convex'' subsets of a sphere will be discussed in
Section \ref{ecc}.
 Section \ref{pinned} contains the rigorous statement of our main result (Theorem \ref{au21.7} restated as Theorem \ref{au21.6}) and its proof, preceded by a few lemmas. Corollary \ref{m5.1}, in the same section, contains explicit upper bounds for the number of collisions for some configurations of balls. 

\section{Foldings}\label{fold}

We will later show that (pseudo-)collisions of a system of pinned balls can be represented as foldings. We will start this section by defining foldings and proving a general result which may have independent interest. We will show that for any sequence of foldings corresponding to a finite family of half-spaces whose intersection has a non-empty interior, the orbit of any point becomes constant eventually. This implies that the number of collisions of pinned balls is finite for any positions of the balls, any initial velocities, and any order in which pairs of balls collide. The proof of this preliminary result is fairly elementary but does not yield an explicit estimate for the maximum number of collisions.

Suppose that $h\in \R^d$ has unit length, and a closed half-space $H\subset \R^d$ is given by
\begin{align*}
H = \{v\in \R^{d } : v \cdot h \geq 0\}.
\end{align*}
It is clear that $\partial H = \left\{ v\in \R^d : v\cdot h = 0 \right\}$, in particular,  $0\in \partial H$. We define a \emph{folding} $\calF_H: \R^d\to \R^d$ relative to $H$ by
\begin{align*}
\calF_H(v) = 
\begin{cases}
v & \text{  if  } v\in H, \\
v - 2 (v\cdot h) h & \text{  if } v\notin H.
\end{cases}
\end{align*}
In other words, $\calF_H$ is the identity on $H$ and it is the reflection in the hyperplane $\prt H$ on the complement of $H$. Note that every folding is non-expansive, i.e., $$\dist\left(\calF_H(x) , \calF_H(y) \right) \leq \dist(x,y)$$ for all $x,y$ and $H$. 

Let $\calB(x, r)$ denote the closed ball with center $x$ and radius $r$ in a  Euclidean space (the dimension of which will be clear from the context). 

A number of steps in the proof of the following result were inspired by the proof of Nagy's Theorem presented in \cite[Thm.~5.3.3]{origami}.

\begin{theorem}\label{au20.1}
Suppose that $H_1, \dots, H_n$ are closed half-spaces in $\R^d$ such that their intersection $H_* :=\bigcap_{k=1}^n H_k$ has a non-empty interior. Assume that $(i_j)_{j\geq 1}$ is a sequence of integers  such that  $1 \leq i_j \leq n$ for all $j$. Suppose that $v_0\in \R^d$ and for $j> 0$, let $v_j = \calF_{H_{i_j}}  (v_{j-1})$.
Then there exists $k<\infty$ such that $v_j = v_k$ for all $j\geq k$. 
\end{theorem}

\begin{proof}

The claim is clearly true if $v_0\in  \bigcap_{k=1}^n \prt H_{k}$ so we will assume in the remaining part of the proof that $v_0\notin  \bigcap_{k=1}^n \prt H_{k}$. 
Fix any (closed) ball $\calB(x,r)$ with $r>0$ that lies in the interior of $H_*$. Let $r_1 = \dist(x, v_0)$ and $A = \calB(x, r_1)$. Since foldings are non-expansive and $x$ is a fixed point for each $\calF_{H_k}$, it follows that $v_j\in A$ for all $j\geq 0$.

Fix the dimension $d$ of the space.
We will use induction on $n$, the number of half-spaces. The theorem is evidently true for $n=1$. Fix any $n>1$ and assume that the theorem holds for any number of half-spaces smaller than $n$.

First suppose that there are $ m \in\{1,\dots, n\}$ and $\ell<\infty$ such that  $(i_j)_{j\geq \ell}$ does not contain any elements equal to $m$.  Then at most $n-1$ half-spaces determine  the evolution of the sequence $v_\ell, v_{\ell+1}, v_{\ell+2}, \dots$. 
By the induction assumption, there exists $k<\infty$ such that $v_j = v_k$ for all $j\geq k$.

From now on we will assume that for every $ m \in\{1,\dots, n\}$, 
the sequence $(i_j)_{j\geq 1}$  contains infinitely many elements equal to $m$.

By compactness of $A$, a subsequence of $(v_j)_{j\geq 0}$ converges to a point, say, $v_\infty \in A$.
For $y\in H_*$, let $f(j, y) = \dist(v_j, y) $. Note that every $y\in H_*$ is invariant under every folding $\calF_{H_k}$. Since foldings are non-expansive, the function $j\to f(j, y)$ is non-increasing  for every $y\in H_*$, and, therefore, $f(\infty, y) := \lim_{j\to \infty} f(j,y)$ exists.
We have $\dist(v_\infty, y) = f(\infty, y)$ for all $y\in H_*$.
We have assumed that $H_*$ has a non-empty interior so there exists at most one  point $v' \in A$ such that $\dist(v', y) = f(\infty, y)$ for all $y\in H_*$. It follows that every convergent subsequence of $(v_j)_{j\geq 0}$ converges to $v_\infty \in A$.
Hence, the whole sequence $(v_j)_{j\geq 0}$ converges to $v_\infty$.

Fix any $k\in \left\{ 1,\ldots,n \right\}$, and let $(i_{k_j})_{j\geq 0}$ be a subsequence of $(i_{j})_{j\geq 0}$ such that $i_{k_j}=k$ for all $j\geq 0$. This is possible because we have assumed that $k$ appears infinitely often in $(i_{j})_{j\geq 0}$. By continuity of $\calF_{H_k}$, it follows that $v_{i_{k_j}+1} = \calF_{H_k}(v_{i_{k_j}})$ converges, as $j\to\infty$,  to both $v_\infty$ and  $\calF_{H_k}(v_\infty)$, which must be the same. Since $k$ is arbitrary, we conclude that $v_\infty\in H_*$.

The set $\bigcap_{k=1}^n \prt H_{k}$ is not empty since it contains zero. Recall that we have assumed at the beginning of the proof that
 $v_0\notin  \bigcap_{k=1}^n \prt H_{k}$. Hence, $\dist(v_0,  \bigcap_{k=1}^n \prt H_{k}) =\delta_1$ for some $ \delta_1>0$. For every $k$, folding $\calF_{H_k}$ preserves the distance between any point in the space and any point in $\prt H_k$ so $\dist(v_j,  \bigcap_{k=1}^n \prt H_{k}) = \delta_1$ for all $j\geq 1$, and, therefore, $\dist(v_\infty,  \bigcap_{k=1}^n \prt H_{k}) = \delta_1>0$. We combine the last observation and the fact that $\dist(v_\infty, H_k) = 0$ for every $k$ to conclude that there are $\eps_1>0$ and $k_*$ such that  $\dist(v_\infty,  H^c_{k_*}) = \eps_1$. 

There exists $m_1< \infty$ such that $\dist(v_j, v_\infty) < \eps_1/2$ for all $j\geq m_1$ and, therefore, $\dist(v_j, H_{k_*}^c) \geq \eps_1/2$ for all $j \geq m_1$. It follows that if $i_j = k_*$ for some $j\geq m_1$ then $v_j = v_{j-1}$. Let $\{\wt H_1, \wt H_2, \dots, \wt H_{n-1}\} = \{H_1, \dots, H_n\} \setminus \{H_{k_*}\}$. 
The sequence $(v_j, j\geq m_1)$ is totally determined by the foldings corresponding to $\wt H_k$'s.
By the induction hypothesis, the theorem holds for $\{\wt H_1, \wt H_2, \dots, \wt H_{n-1}\}$ with $v_0$ replaced by $v_{m_1}$, so there exists $\ell\geq m_1$ such that $v_j = v_\ell$ for all $ j\geq \ell$.
\end{proof}

\begin{remark}\label{s28.6}
Recall notation from the statement of Theorem \ref{au20.1}. The set of all distinct elements of the sequence $(v_j, j\geq 1)$ will be called an orbit. Note that the orbit depends on  half-spaces $H_k$, sequence $(i_j, j\geq 1)$ and $v_0$. In view of the applications of the concept of folding to collisions of hard balls (discussed later in the paper), it is of interest to ask whether there is an upper bound on the size (cardinality) of the orbit that 
depends only on $n$, that is, the number of half-spaces $H_k$. Such a bound would give an upper bound for the maximum number of collisions of pinned balls depending only on the number of balls.

Consider half-planes $H_1$ and $H_2$ defined in complex notation by $H_j=\{w \in \R^2: w \cdot e^{i \theta_j}\geq 0\}$, for $j=1,2$. It is easy to see that for any $m$, one can generate orbits with more than $m$ points by choosing $\theta_1$ and $\theta_2$ so that $|(\theta_1-\theta_2) - \pi|$ is non-zero but very small.
Hence, a universal bound for the orbit size depending only on $n$ does not exist. 

\end{remark}

\section{Pinned balls: preliminaries}\label{prelim}

All vectors should be interpreted as column vectors, even if they are written as row vectors, for typographical convenience. This convention will matter only in those arguments where we collect vectors to form a matrix. 

Recall that $\calB(x, r)$ denotes the closed ball with center $x$ and radius $r$ and let $S= \prt \calB(0,1) $.

We will consider a family of $n$ balls $\bF=\{\calB(x_1,1), \dots,\calB(x_n,1)\}$ in $\R^d$, for $d\geq 1$ and $n\geq 3$. We will assume that the interiors of the balls are disjoint but the balls may touch, i.e., for some pairs of balls, the distance between their centers is equal to 2.

We will say that $(\calV, \calE)$ is the \emph{full graph} associated with the family $\bF$ of $n$ pinned balls if   $\calV=\{x_1,x_2, \ldots , x_n\}$,  and   vertices $x_j$ and $x_k$ are connected by an edge if and only if $|x_j-x_k|=2$, i.e., if the balls $j$ and $k$ touch. An edge connecting $x_j$ and $x_k$ will be denoted $(j,k)$.

We will say that $(\calV_1, \calE_1)$ is a \emph{graph}  associated with the family $\bF$ if $(\calV_1, \calE_1)$ is a subgraph of the full graph associated with $\bF$.

Note that a full graph is not a complete graph in the graph-theoretic sense unless the  centers of the balls form the vertex set of a simplex.


We associate a pseudo-velocity $v_k\in \R^d$ to the $k$-th ball, for $k=1,\dots, n$ (note that from now on, the meaning of $v_k$ will be different from that in Section \ref{fold}). We call $v_k$ a pseudo-velocity because the balls do not move---their centers, i.e., $x_k$'s, are fixed. However, the pseudo-velocities will change due to pseudo-collisions as in an evolution of a family of billiard balls with totally elastic collisions. 
We will now define formally a pseudo-collision as a mapping $\calT_{ij}: \R^{nd} \to \R^{nd}$ for  $1\leq i,j\leq n$. 

Let $ v  = (v_1, v_2,\dots, v_n)\in \R^{nd}$ and let 
\begin{align}\label{au20.3}
\Pi_k( v ) = v_k
\end{align}
 for $k=1,\dots, n$. 

It will be convenient to write $\calT_{ij}( v ) =  w  =(w_1,\dots,w_n)$, with $w_k\in\R^d$ for every $k$, so that we can define $\calT_{ij}$ by specifying the values of $w_k$'s.

First, we let $w_k = v_k$ for every $k\ne i,j$. In other words, a collision between balls $i$ and $j$ does not affect the velocity of any other ball.

If the balls  $i$ and $j$ do not touch (i.e., $|x_i-x_j| > 2$) then we let $w_k = v_k$ for all $k$. Heuristically, balls which do not touch cannot collide.

If the balls  $i$ and $j$  touch and
\begin{align*}
(v_{i} - v_{j}) \cdot (x_{i} - x_{j}) \geq  0
\end{align*}
then once again there is no collision, i.e., we let $w_k = v_k$ for all $k$.

Finally, assume that the balls  $i$ and $j$  touch, i.e., $|x_i-x_j| = 2$, and
\begin{align*}
(v_{i} - v_{j}) \cdot (x_{i} - x_{j}) <  0.
\end{align*}
Let $u  = (x_{i} - x_{j})/|x_{i} - x_{j}|$.
Then we let
\begin{align}\label{a23.2}
w_{i} &= v_{i} + (v_{j} \cdot u) u
- (v_{i} \cdot u) u,\\
w_{j} &= v_{j} + (v_{i} \cdot u) u
- (v_{j} \cdot u) u.\label{a23.3}
\end{align}
In other words, the balls exchange the components of their pseudo-velocities that are parallel to the line through their centers; the orthogonal components  remain unchanged. This rule is identical to the classical totally elastic collision.

\medskip

We will now represent  pseudo-collisions as  foldings.
For $k=1,\dots,n$ and $x\in \R^d$, let
\begin{align*}
x_{[k]} = ( \underbrace{0,\dots,0}_{(k-1)d}, x, \underbrace{0,\dots,0}_{(n-k)d} ) \in \R^{nd}.
\end{align*}
Let $\calE$ be the edge set for the full graph representing a family $\bF$. For $(j,k)\in \calE$, we let 
\begin{align}
\wt  z _{jk} &:= ( x_j - x_k)_{[j]} + (x_k - x_j)_{[k]} \in \R^{nd},
\label{j14.1}\\
  z _{jk} &:= \wt  z _{jk}/|\wt  z _{jk}| = 2^{-3/2}\wt  z _{jk} .
\label{au20.2}
\end{align}
Let $H_{jk}= \{ w  \in \R^{nd}:  w  \cdot  z _{jk} \geq 0\}$. In the case of a pseudo-collision, the transformation $\calT_{ij}$ defined in \eqref{a23.2}-\eqref{a23.3} is the same as the folding  $\calF_{H_{i j}}$ in $\R^{nd}$; we state this remark as a numbered formula for future reference,
\begin{align}\label{au21.2}
\calT_{ij}=\calF_{H_{i j}}.
\end{align}
Note that $z_{ij} = z_{ji}$, hence $H_{ij} = H_{ji}$.

\medskip

From now on, we will use $t$ to denote integer valued ``time'' parameter.

Suppose that $(\calV_1, \calE_1)$ is a graph associated with a family  of pinned balls. Consider a sequence $\Gamma=(\gamma_j, j\geq 1)$, such that $\gamma_j \in \calE_1$ for every $j$. 
If $\gamma_j$ is an edge connecting vertices $x_i$ and $x_m$ then $\calT_{\gamma_j}$ should be interpreted as $\calT_{im}$.
For $v\in \R^{nd}$, let $v(0) = v$ and $v(t) = \calT_{\gamma_j}( v(t-1))$ for $t\geq 1$. 
The sequence $\Gamma$ represents the exogenous order of collisions for the system of pinned balls.

Note that  $v(t) = v(t-1)$ for some $\bF, v, \calE_1,\Gamma$ and $t$, i.e., a pseudo-collision does not have to change the velocities. 

We will use the notation $v(t) = (v_1(t), \dots, v_n(t))$, i.e., $v_i(t)$ is the pseudo-velocity of ball $i$ at time $t$.

Let $\Lambda(\bF, v, \calE_1,\Gamma)$ denote the number of $t$ such that $v(t) \ne v(t-1)$. 
Let $L(m,\bF)$ denote the supremum of $\Lambda(\bF, v,\calE_1, \Gamma)$ taken over  all $v\in \R^{dn}$, all $\calE_1 $ with $|\calE_1| = m$ such that $(\calV_1, \calE_1)$ is a graph associated with $\bF$, and all $\Gamma$ with values in $\calE_1$. Note that, implicitly, $L(m,\bF)$ depends on $n$ and $d$.

\subsection{A monotone functional}

The following  functional was introduced and used to study hard ball collisions in \cite{Vaser79,Illner89, Illner90,IllnerChen}. 
Let
\begin{align}
\label{aug9.2}
F(t)=
\sum_{i,j=1}^n  (v_j(t) - v_i(t)) \cdot (x_j - x_i).  
\end{align}

From now on, we will assume that the center of mass and the total momentum are 0 and the total energy is 1. More accurately,
\begin{align}\label{au21.1}
\sum_{j=1} ^n x_j =0 , \qquad \sum_{j=1} ^n v_j =0 ,
\qquad \sum_{i=1}^n |v_i|^2 = 1.
\end{align}
We make these assumptions because they allow us to simplify some expressions. We are mainly interested in the number of (pseudo-)collisions.
None of the following operations will change the number of collisions experienced by a system of pinned balls: (i) adding a constant vector to all $x_j$'s; (ii) adding a constant vector to all $v_j$'s; (iii) multiplying all $v_j$'s by the same scalar $c\ne 0$.

\begin{lemma}\label{au21.3} Let $x=(x_1, \dots, x_n)\in\R^{nd}$. Then $F(t) =2n x\cdot v (t)$. Moreover, if $\gamma_t = (i,j)$ then
\begin{align}\label{eq:jump_f}
F(t)-F(t-1) = 4n |v_i(t)-v_i(t-1)|.
\end{align}
If in addition there is a collision at time $t$, i.e., $v_i(t-1) \ne v_i(t)$, then
\begin{align*}
F(t)-F(t-1) = 2n (v_j(t-1)-v_i(t-1))\cdot (x_i-x_j).
\end{align*}

We can conclude that the functional $F$ is non-decreasing.

\end{lemma}
\begin{proof}
The first equation follows from elementary sum manipulation and \eqref{au21.1}.
If there is no collision at time $t$, i.e., $v(t-1) = v(t)$, then
\eqref{eq:jump_f} is obviously true. 

Assume that there is a collision at time $t$.
If $\gamma_t = (i,j)$ then, by part (i),
\begin{align*}
F(t) - F(t-1) 
= 2n (v_i(t)-v_i(t-1))\cdot x_i + 2n (v_j(t)-v_j(t-1))\cdot x_j.
\end{align*}
By the definition of $\calT_{ij}$, we have $v_i(t)-v_i(t-1) = -(v_j(t)-v_j(t-1))$. Since the change of velocity is only along the direction $x_i-x_j$, we obtain
\begin{align}
\label{oc2.51}
 F(t) - F(t-1) &= 2n (v_i(t)-v_i(t-1))\cdot (x_i-x_j) 
= 4n |v_i(t)-v_i(t-1)| .
\end{align}

(iii) It follows from \eqref{a23.2}-\eqref{a23.3} that 
$v_i(t)\cdot (x_i-x_j) = v_j(t-1)\cdot (x_i-x_j)$
so
\begin{align*}
(v_i(t)- v_j(t-1))\cdot (x_i-x_j)=0.
\end{align*}
This and \eqref{oc2.51} imply that
\begin{align*}
 F(t) - F(t-1) &= 2n (v_i(t)-v_i(t-1))\cdot (x_i-x_j) \\
&= 2n (v_j(t-1)-v_i(t-1))\cdot (x_i-x_j)
+ 2n (v_i(t)- v_j(t-1))\cdot (x_i-x_j)\\
&=  2n (v_j(t-1)-v_i(t-1))\cdot (x_i-x_j).
\end{align*}

(iv) This part follows from (ii).
\end{proof}

\subsection{Kissing number}\label{kiss}

The  kissing number  of a $d$-dimensional ball, i.e.,  the maximum
number of mutually nonoverlapping translates of the ball that can be arranged
so that they all touch the ball, will be denoted $\tau_d$. 
According to  \cite[Thm.~1.1.3]{Bez}, 
\begin{align*}
2^{0.2075d(1+ o(1))} \leq \tau_d \leq 2^{0.401d(1+o(1))}.
\end{align*}
An elementary non-asymptotic bound is  $2d\leq \tau_d \leq 3^d-1$.

\section{Approximate rigidity}\label{s28.1}

We will define the ``index of approximate rigidity'' $\alpha(\bF)$. The name will be explained after the definition.

\begin{definition}\label{s21.1}
Consider a family of pinned balls $\bF$ and recall the definition of $z_{jk}$ from  \eqref{au20.2}.

Consider a graph $G_1=(\calV_1, \calE_1)$ associated with $\bF$, suppose that $(i_1, i_2) \in \calE_1$ and let $\calE_2 = \calE_1 \setminus \{(i_1, i_2)\}$.
Let $\alpha_*(G_1, (i_1, i_2))$ be the distance from $  z _{i_1 i_2}$ to the linear subspace spanned by $\left\{ z _{kj}, (k,j) \in \calE_2\}\right\}$.

We will write $\alpha(\bF)$ 
to denote the minimum of all \emph{strictly positive} values of 
$\alpha_*(G_1, (i_1, i_2))$ over all graphs $G_1=(\calV_1, \calE_1)$ associated with $\bF$ and all $(i_1, i_2)\in \calE_1$.

\end{definition}

\begin{remark}\label{s24.20}
(i) Definition \eqref{s21.1} emphasized that  $\alpha(\bF)$ 
is the minimum of all strictly positive values of 
$\alpha_*(G_1, (i_1, i_2))$ but it is worth repeating the point---the values
of $\alpha_*(G_1, (i_1, i_2))$ equal to zero are excluded from the minimum in the definition of $\alpha(\bF)$. We will explain when $\alpha_*(G_1, (i_1, i_2))$ may be equal to zero in  part (ii) of the remark. We will present some examples illustrating the cases when $\alpha_*(G_1, (i_1, i_2))$ is zero or (non-zero and) close to zero in Figures \ref{fig1} and \ref{fig2}.
 
It is easy to check that if $\calV_1$ has three elements then $\alpha_*(G_1, (i_1, i_2))>0$.
Since the number $n$ of balls is finite, it follows from the definition that 
$\alpha(\bF)>0$.

(ii) 
Recall the notation from Definition \ref{s21.1}. The calculations presented in the proof of Lemma \ref{m3.1} below indicate that $\alpha_*(G_1, (i_1, i_2))=0$ if there exist $a_{jk}\in \R$ for all $(j,k) \in \calE_1$ such that $a_{i_1 i_2} =1$, $a_{jk} = a_{kj}$ for all $(j,k) \in \calE_1$, and for every $k \in \calV_1$,
\begin{align}\label{s16.1}
\sum_{(j,k) \in \calE_1} a_{jk} (x_j - x_k) = 0.
\end{align}
This condition  has the following physical interpretation. 
Suppose that we place a rigid rod between $x_j$ and $x_k$ for  every $(j,k) \in \calE_1$. The rods are joined at points  $x_j$ by totally flexible hinges.
If we replace the rod between $x_{i_1}$ and $x_{i_2}$ by a spring which exerts a (positive or negative) non-vanishing force $a_{i_1i_2} (x_{i_1} - x_{i_2})$
on $x_{i_1}$ and the symmetric force on $x_{i_2}$ then the points $x_k\in \calV_2$ will not move because there is a family of forces $a_{jk} (x_j - x_k)$ which balance each other at every $x_k$, by \eqref{s16.1}.
The system of rods is ``infinitesimally rigid'' in this sense. An infinitesimally rigid graph associated with pinned discs is shown in Figure \ref{fig1}.  See \cite[Chap.~4]{origami} for a discussion of rigidity, stress (i.e., the family of forces mentioned above) and the relationship between these concepts.

\begin{figure} \includegraphics[width=0.4\linewidth]{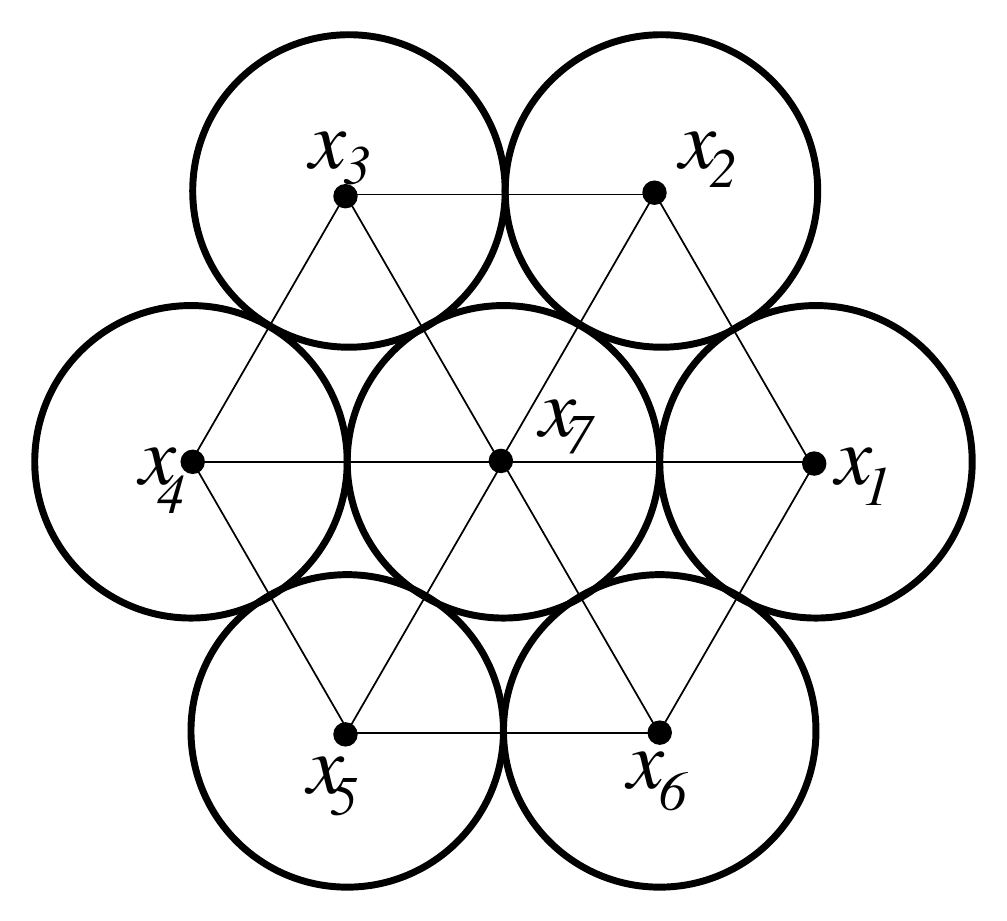}
\caption{
The graph associated with this system of pinned discs is rigid. See Remark \ref{s24.20} for the discussion of this concept.
}
\label{fig1}
\end{figure}

We call $\alpha(\bF)$ the
``index of approximate rigidity'' because it measures, in a sense, how close to being infinitesimally rigid are those 
graphs associated with $\bF$ which are not infinitesimally rigid. Figure \ref{fig2} presents a family of pinned discs with a very small index of approximate rigidity. These types of ball configurations present the greatest challenge for our methods; our estimates of the maximum number of collisions are, most likely, far from being sharp for pinned ball families with a very small index of approximate rigidity. 

\begin{figure}
    \centering
    \subfloat{{\includegraphics[width=0.4\linewidth]{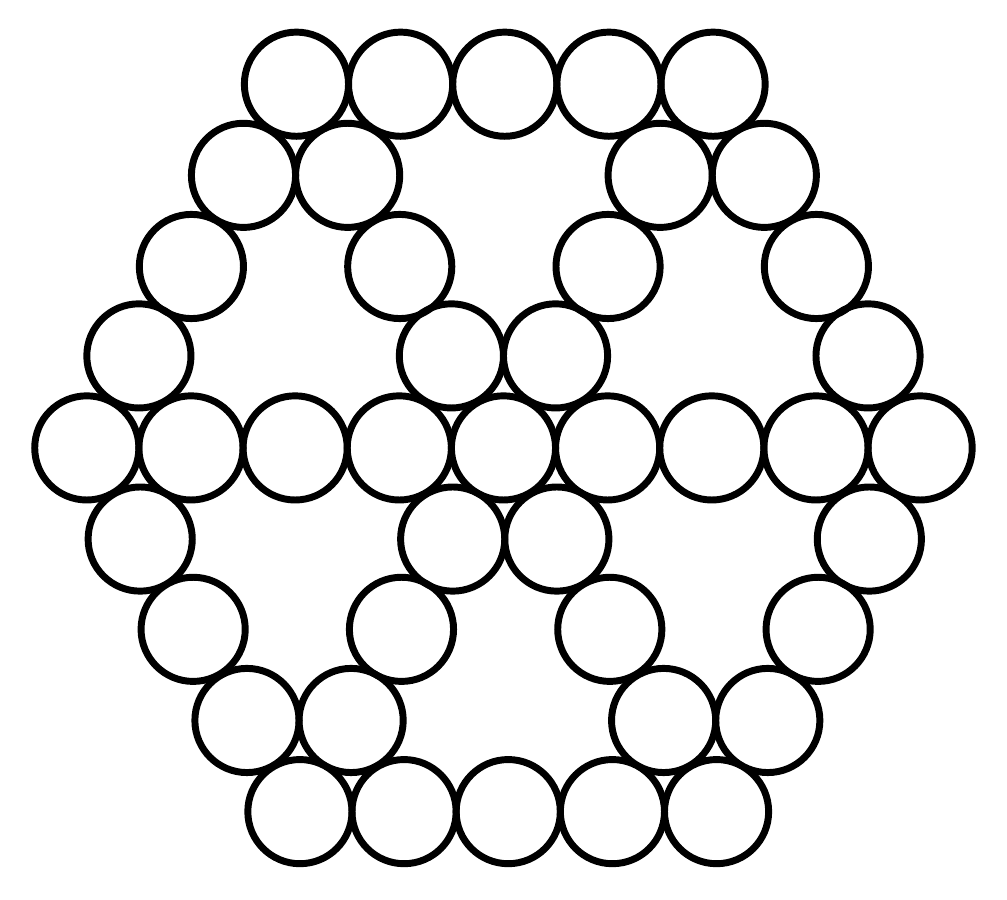} }}
    \qquad
    \subfloat{{\includegraphics[width=0.4\linewidth]{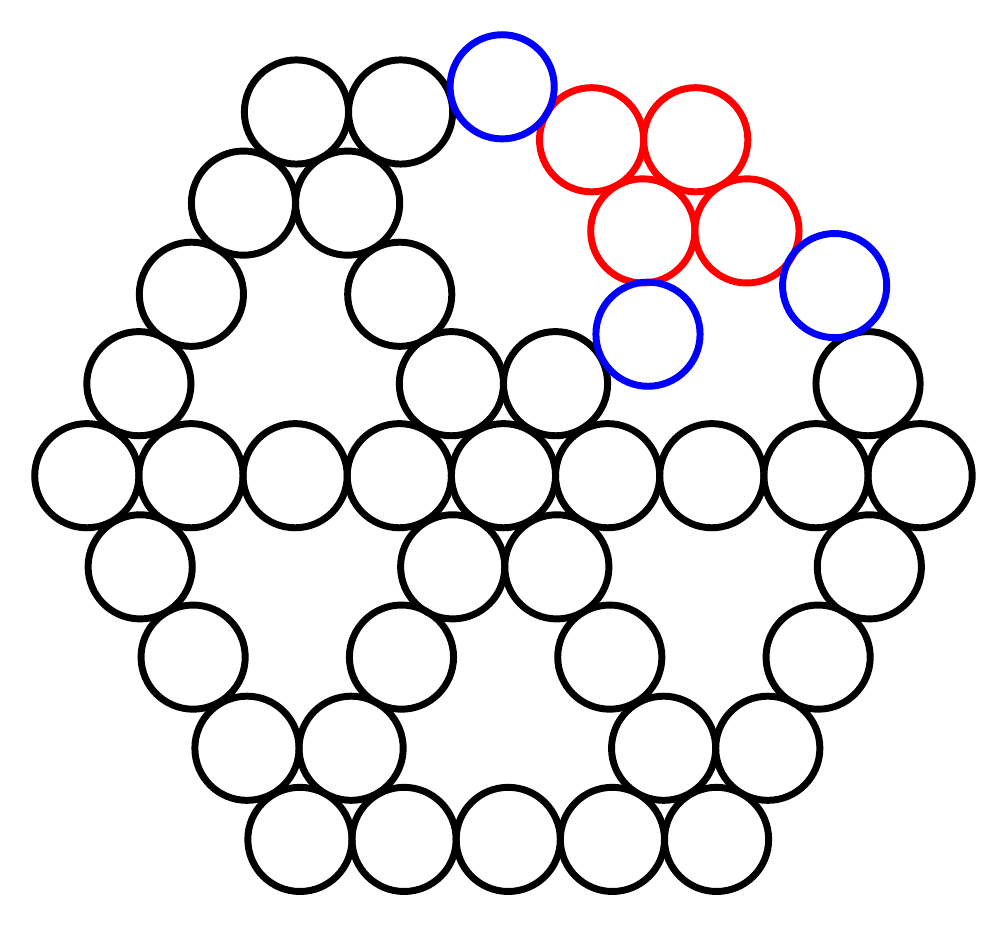} }}
    \caption{The graph associated with the family of pinned discs on the left is infinitesimally rigid.
The configuration on the right hand side was obtained from the one on the left hand side by moving the red discs slightly towards the central disc and then adjusting positions of the blue discs so that they touch the same discs as in the original configuration. The graph on the right is not infinitesimally rigid but it is almost infinitesimally rigid in the sense that for every $\eps>0$ one can make the displacement of the red discs sufficiently small so that it is possible to find coefficients $a_{ij}$ such that the norm of the sum in \eqref{s16.1} is less than $\eps$ for all $k$.}
    \label{fig2}
\end{figure}

(iii) If points $\{x_1, \dots , x_n\}$ are given explicitly then $\alpha(\bF)$ can be calculated using explicit formulas---some of the steps of such a calculation will appear in the proof of Lemma \ref{s26.10}. We could not find simple or intuitive estimates for $\alpha(\bF)$ for an arbitrary set of $x_j$'s.
We will present estimates for $\alpha(\bF)$ in two specific cases.
First, we will consider families of pinned balls such that the corresponding graphs are trees. Then we will consider balls in $\R^2$ with centers at the vertices of the usual regular triangular lattice. In other words, we will consider finite subfamilies of the tightest disc packing in the plane.
 We believe that both families of pinned balls have intrinsic interest and our estimates shed light on the magnitude of $\alpha(\bF)$ in each case.

(iv) Our main result, an upper estimate for  the number of collisions of pinned balls, is based on a lower bound for $\alpha(\bF)$. Hence, this is the bound that we will focus on.
\end{remark}

\begin{lemma}\label{m3.1}
Suppose that the full graph $G=(\calV, \calE)$  associated to a family of $n$ pinned balls  is a tree. Then $\alpha(\bF) \geq 4/n$.

\end{lemma}

\begin{proof}

Consider a subgraph $G_1=(\calV_1, \calE_1)$ of $G$, an edge $(i_1, i_2) \in \calE_1$ and let  $\calE_2 = \calE_1 \setminus \{(i_1, i_2)\}$.
Consider any family of real numbers $\left\{a_{k,m}, (k,m) \in \calE_2\right\}$, and let
\begin{align}\label{a25.1}
 z _* = \sum_{(k,m) \in \calE_2} a_{k,m}  z _{k m}.
\end{align}

Recall that $ z _{jk} =  z _{kj}$ for all $j$ and $k$. If we replace the coefficients of $z_{mk}$ and $z_{km}$ with $ \frac12(a_{k,m}+ a_{m,k}) $ then this will not change $z_*$. For this reason, we do not lose generality by assuming that $a_{j,k} =a_{k,j}$. 

It will suffice to prove that
\begin{align}\label{a25.2}
\left|  z _{i_1 i_2} -  z _*\right| \geq 4/n . 
\end{align}

Since $G$ is a tree and $G_2= (\calV_1,\calE_2)$ is a  subgraph of $G$, it follows that $G_2 $ is a forest, i.e., a disjoint union of trees. Vertices $x_{i_1}$ and $x_{i_2}$ belong to different trees in the forest $G_2$. Let $G_3 = (\calW_3,\calE_3)$ be the maximal tree in $G_2$ such that $x_{i_1}\in \calW_3$. Define analogously $G_4 = (\calW_4,\calE_4)$ relative to $x_{i_2}$. Assume that the cardinality of $\calW_3$ is less than or equal to the cardinality of $\calW_4$; otherwise, exchange their labels.

\medskip
Recall the notation from \eqref{au20.3} and note that
\begin{align}\label{s21.2}
 \Pi_{i_1}(  z _{i_1 i_2}) - \Pi_{i_1}(  z _*)
= (x_{i_1} - x_{i_2}) - \sum_{(i_1,j) \in \calE_3}
a_{i_1,j} (x_{i_1} - x_{j}).
\end{align}
If $ x_k\in \calW_3$ and  $k\ne i_1$  then,
\begin{align}\label{s21.3}
 \Pi_{k}(  z _{i_1 i_2}) - \Pi_{k}(  z _*)
=  - \sum_{j:(k,j) \in \calE_3}
a_{k,j} (x_{k} - x_{j}).
\end{align}
We combine \eqref{s21.2}-\eqref{s21.3} to obtain
\begin{align*}
&\sum_{ k: x_k\in \calW_3} \left(
 \Pi_{k}(  z _{i_1 i_2}) - \Pi_{k}(  z _*)\right)
=\Pi_{i_1}(  z _{i_1 i_2}) - \Pi_{i_1}(  z _*)
+
\sum_{ k:x_k\in \calW_3\setminus \{ x_{i_1} \}} \left(
 \Pi_{k}(  z _{i_1 i_2}) - \Pi_{k}(  z _*)\right)\\
& =
 (x_{i_1} - x_{i_2}) - \sum_{j:(i_1,j) \in \calE_3}
a_{i_1,j} (x_{i_1} - x_{j})
+
\sum_{ k:x_k\in \calW_3\setminus \{ x_{i_1} \}} \left(
 - \sum_{j:(k,j) \in \calE_3}
a_{k,j} (x_{k} - x_{j})\right)\\
&= (x_{i_1} - x_{i_2}) - \sum_{(k,j) \in \calE_3} a_{k,j} (x_k - x_j) = (x_{i_1} - x_{i_2}) .
\end{align*}
The last equality  holds because the  terms in the sum are products of  symmetric scalars $a_{k,j}=a_{j,k}$ and  antisymmetric vectors $x_k - x_j=-(x_j-x_k)$.
Thus
\begin{align*}
\sum_{ k: x_k\in \calW_3} \left|
 \Pi_{k}(  z _{i_1 i_2}) - \Pi_{k}(  z _*)\right|
&\geq \left|
\sum_{ k: x_k\in \calW_3 } \left(
 \Pi_{k}(  z _{i_1 i_2}) - \Pi_{k}(  z _*)\right)\right| = |x_{i_1} - x_{i_2}|=2.
\end{align*}
The number of summands on the left hand side is bounded by floor of $n/2$ (because of our assumption on the cardinality of $\calW_3$) so one of the summands is greater than or equal to $4/n$. Hence, for some $ k\in \calW_3$,
\begin{align*}
 \left|
  z _{i_1 i_2} -   z _*\right|
\geq
 \left|
 \Pi_{k}(  z _{i_1 i_2}) - \Pi_{k}(  z _*)\right|
\geq 4/n.
\end{align*}
We have proved \eqref{a25.2}.
\end{proof}

We will need the following elementary linear algebra lemma, which we present without proof.

\begin{lemma}\label{s23.1}
Suppose that $m> 1$ and $(e_j, 1\leq j \leq m)$ is the usual orthonormal basis of $\R^{m}$. Suppose that $w_j\in \R^{m}$ for $ 1\leq j \leq k$ where $k <m$. Assume that $w_j$'s are linearly independent. Let $V$ be the linear subspace $V\subset \R^{m}$ spanned by $w_j$'s. Assume that $w \in \R^{m} \setminus V$. Then there exists a subset   $(f_j, 1\leq j \leq r)$ of $(e_j, 1\leq j \leq m)$ for some $1\leq r \leq m$ such that $f_j$'s are linearly independent of $w_j$'s and if $W$ denotes the space spanned by $w_j$'s and $f_j$'s then $W$ is $(m-1)$-dimensional and $w \notin W$.

\end{lemma}


We will now study planar families of pinned discs  which are subsets of the densest packing of discs in the plane. For future reference we state the following as a formal definition.

\begin{definition}\label{s27.1}
Let $\calX$ be the set of all points in the plane of the form
$( 2j, 2 k \sqrt{3})$ or $( 2j+1, (2 k+1) \sqrt{3})$ for integers $j$ and $k$. That is, $\calX$ is the set of vertices of a triangular lattice, assuming that we add edges between pairs of points at distance 2.
\end{definition}

The family of all unit discs with centers in $\calX$ is the densest packing of unit discs in the plane (see \cite{Toth}). It is natural to consider a ``tight'' subfamily of these discs, for example, all those that fit into a given large disc.  However, our estimates do not depend on any particular arrangement of the discs with centers in $\calX$. 

\begin{lemma}
\label{le:m.8}
Consider an $m\times m$ matrix $A$, for some $m\geq 1$. Assume that each of its columns satisfies one of the following conditions.
\begin{enumerate}[(a)]
\item Exactly one  component is non-zero and it is equal to $1$.
\item Exactly two  components are non-zero,  one of these is equal to $2$ and the other one is equal to $-2$.
\item Exactly four  components are non-zero, two of which have absolute value $1$, and the other two have absolute value $\sqrt{3}$.
\end{enumerate}
Then the determinant of $A$ has the form $r_1+r_2\sqrt{3}$, where $r_1$ and $r_2$ are integers and 
\begin{align}
\label{eq:m.8}
| r_1 | \leq 4^m,\qquad | r_2 | \leq 4^m/\sqrt{3}.
\end{align}
It follows that $|\det A |\leq 2\cdot 4^m$.
\end{lemma}
\begin{proof}
The form $r_1+r_2\sqrt{3}$ of  $\det A$ is given by the fact that $\mathbb Z[\sqrt 3]$ is a ring, and the entries of $A$ are in $\mathbb Z[\sqrt 3]$. 
 
Note that \eqref{eq:m.8} is invariant under swapping of columns of $A$. Hence, we will  assume without loss of generality that the first $m_1$ columns of $A$ satisfy $(a)$; the following $m_2$ columns of $A$ satisfy $(b)$; and the remaining $m_3$ columns of $A$ satisfy $(c)$. 

To simplify the exposition, we will say that $B$ is a sub-matrix of $A$ if $B$ is a square matrix obtained from $A$ by removing some columns and rows of $A$, and by changing the sign of some (possibly none) of the resulting columns. 
We will use this generalized definition of a sub-matrix to avoid  keeping track of  signs when expanding determinants by minors. 

We start by expanding  $\det A$ by cofactors of its first $m_1$ columns. Each of these columns is a vector of the canonical basis, and thus $\det A$ equals to $\det A'$, where $A'$ is an $(m-m_1)\times (m-m_1)$ sub-matrix of $A$. Next, we expand $\det A'$ by cofactors of its first $m_2$ columns. Each of these columns can have at most two non-zero components because the respective columns of $A$ satisfy (b) but  some of the non-zero components present in $A$ may have been removed since $A'$ is a sub-matrix of $A$. If some of these columns have only zero components then $\det A' =0$ and, therefore, \eqref{eq:m.8} holds. Otherwise, we have a non-trivial expansion by minors using the first $m_2$ columns of $A'$. It follows that for some $N\leq 2^{m_2}$, there are matrices $A_1,\ldots,A_N$ such that 
\begin{align}
\label{eq:j.12.01}
\det A = 2^{m_2} \sum_{i=1}^{N} \det A_i,
\end{align}
where, for $i=1,\ldots, N$:
\begin{itemize}
\item $A_i$ is an $m_3\times m_3$ submatrix of $A$.
\item Each column of $A_i$ has at most four non-zero components and the absolute values of these components can only take values $1$ or $\sqrt{3}$. At most two components may take absolute value $1$, and at most two components may take absolute value $\sqrt{3}$.
\end{itemize}

We will next use the multi-linearity of the determinant to analyze the $m_3$ columns of each $A_i$. The properties of each $A_i$ listed above allow us to represent the $j$-th column of $A_i$ as $a^i_j+b^i_j\sqrt{3}$ where each of $a^i_j$ and $b^i_j$ is a vector with at most two non-zero  components which can only be $1$ or $-1$. The number of non-zero components (zero, one or two) in $a^i_j$ or $b^i_j$ depends on which rows of $A$ have been removed to obtain the sub-matrix $A_i$.

Let $\lambda^k$ be an ordering of the elements of $\{0,1\}^{m_3}$, and denote $|\lambda^k| = \sum_{j=1}^{m_3} \lambda^k_j$. 
Let $A^k_j$ be an $m_3 \times m_3$ matrix  whose $j$-th column is $a^i_j$ if $\lambda^k_j = 0$, and $b^i_j$ otherwise.  By multi-linearity of the determinant, we have
\begin{align}
\label{eq:ju17.1}
\det A_i   
=  \sum_{k=1}^{2^{m_3}} \left(\sqrt{3}\right)^{|\lambda^k|} \det A_i^k.
\end{align}
Note that some vectors $a^i_j$ or $b^i_j$ may be identically zero so some matrices $A^k_i$ may have zero determinant.
Formulas  \eqref{eq:j.12.01} and  \eqref{eq:ju17.1}  imply that
\begin{align}
\label{eq:j13.01}
\det A &=  2^{m_2} \sum_{i=1}^{N} \sum_{k=1}^{2^{m_3}} \left(\sqrt{3}\right)^{|\lambda^k|} 
\det A_i^k 
=  2^{m_2}  \sum_{q=0}^{m_3} \left(\sqrt{3}\right)^{q} \sum_{i=1}^{N}  \sum_{k : |\lambda^k| = q }\det A_i^k.
\end{align}
This and the representation $\det A = r_1 + r_2 \sqrt{3}$ imply that
\begin{align*}
r_1 &=  2^{m_2}  \sum_{ \underset{q \text{ is even}}{q=0}}^{m_3} \left(\sqrt{3}\right)^{ q} \sum_{i=1}^{{N}}  \sum_{k : |\lambda^k| = q } \det A_i^k , \qquad
r_2 =   2^{m_2}  \sum_{ \underset{q \text{ is odd}}{q=0}}^{m_3} \left(\sqrt{3}\right)^{ q-1} \sum_{i=1}^{N}  \sum_{ k : |\lambda^k| = q } \det A_i^k . 
\end{align*}
Recall that the columns of $A^k_i$ have at most  two non zero components, which can only be $1$ or $-1$. It follows that  $\det A^k_i$ is an integer satisfying $| \det A^k_i| \leq 2^{m_3/2}$ by Hadamard's inequality (see \cite[p.~64]{BB}). We obtain
\begin{align*}
| r_1 | &\leq  2^{m_2}  \sum_{ \underset{q \text{ is even}}{q=0}}^{m_3} \left(\sqrt{3}\right)^{ q} \sum_{i=1}^{N}  \sum_{k : |\lambda^k| = q } 2^{m_3/2} 
\leq 2^{m_2}  \sum_{ \underset{q \text{ is even}}{q=0}}^{m_3} \left(\sqrt{3}\right)^{ q} {N}   {m_3 \choose q} 2^{m_3/2}   \\
&\leq 2^{2m_2 + m_3/2} (1+\sqrt{3})^{m_3} \leq 2^{2m_2 + m_3/2} 2^{3{m_3} /2 } \leq 4^m.
\end{align*}
This proves the first inequality in \eqref{eq:m.8}. The other inequality can  be proved analogously.
\end{proof}

\begin{lemma}\label{s26.10}
Suppose that 
$d=2$ and the centers of $n$ discs in the family $\bF$  belong to $\calX$, as in Definition \ref{s27.1}. Then
\begin{align*}
\alpha(\bF) \geq \frac{\sqrt{3}}{432}\cdot\frac{1}{4^{4n}\sqrt{n}}.
\end{align*}
\end{lemma}

\begin{proof}

Consider  graphs $G_k=(\calV_k, \calE_k)$, $k=1,2$, associated with $\bF$, such that $\calV_1=\calV_2$ and $\calE_2 = \calE_1 \setminus \{(i_1, i_2)\}$ for some $(i_1, i_2) \in \calE_1$.

We will estimate the distance from $z_{i_1i_2}$ to the subspace $V:=\mathrm{span}\{z_{ij}, (i,j) \in \calE_2\}$. 
If the distance is zero then we can set this case aside, in view of Definition \ref{s21.1}, where the minimum is taken over strictly positive values of $\alpha_*$. In the rest of the proof, we will assume that the distance is strictly positive.

Recall notation from \eqref{j14.1}.
Fix a maximal linearly independent subset of $\{\wt z_{ij}, (i,j) \in \calE_2\}$, and call its elements $\wh z_1, \wh z_2, \dots, \wh z_{p}$. This is a basis for $V$, and $p=\dim V$. We use Lemma \ref{s23.1} to find elements $e_{m_1}, e_{m_2}, \dots, e_{m_{q}}$ of the usual orthonormal basis of $\R^{2n}$ so that the vectors in the set $\{\wh z_1, \wh z_2, \dots, \wh z_{p},e_{m_1}, e_{m_2}, \dots, e_{m_{q}}\}$ are linearly independent and span a $(2n-1)$-dimensional subspace $W$, such that $z_{i_1i_2}\notin W$. Since $V\subset W$, we have $0<\dist(z_{i_1i_2}, W) \leq \dist(z_{i_1i_2}, V)$.

\newcommand{\bc}{\mathbf{c}}

For $w\in \R^{nd}$, let $M(w)$ be the $2n\times 2n$ matrix whose columns are the vectors $w,\wh z_1,\ldots,\wh z_p$, and $e_{m_1},\ldots,e_{m_q}$, in this order.  We have  $w\in W$ if and only if
\begin{align}
\label{detw}
\det M(w) = 0.
\end{align}
We will next expand the determinant of $M(w)$ by cofactors of the first column. Let $C_i$ be the $(i,1)$-cofactor of the matrix $M(w)$. Since the first column is ignored when computing this cofactor, it follows that $C_i$ does not depend on $w$. Moreover, we have $C_i = \det M(e_i)$. By setting $\mathbf{c} = (C_i)_{i=1,\ldots, 2n}$, and using the expansion by cofactors, we can rewrite \eqref{detw} as $w\cdot\mathbf{c} =0$, which shows that $\bc$ is a normal vector to $W$. Hence, we obtain the bound
\begin{align}\label{s24.2}
\dist(z_{i_1i_2}, V) \geq \dist(z_{i_1i_2}, W)
= \frac {|z_{i_1i_2}\cdot \bc|}{|\bc|} .\end{align}

We will next estimate $|\bc|$. First, note that for $j=1,\ldots,q$, the matrix $M(e_{m_j})$ has two identical columns, and its determinant is zero. It follows that $\bc$ has at most $2n-q = p+1$ nonzero components. 

We will argue that every matrix of the form $M(e_i)$ or  $M(\wt z_{i_1i_2})$
satisfies the assumptions of Lemma \ref{le:m.8}. Columns of the form $e_i$ or $e_{m_j}$ satisfy assumption (a). The first column of $M(\wt z_{i_1i_2})$ and columns $\wh z_j$ must have the form $\wt z_{j_1,j_2}$. These correspond either to touching balls with centers on a horizontal line, in which case (b) is satisfied, or with centers on a line with slope $\pm \pi/3$, in which case (c) is satisfied.

Lemma \ref{le:m.8}  implies that $|C_i| =|\det M(e_i)|\leq 2\cdot 4^{2n}$ for all $i$, so
\begin{align}\label{s24.1}
|\bc| \leq  2 \cdot 4^{2n}\sqrt{p+1}.
\end{align}

We next find a lower bound for $|z_{i_1i_2}\cdot \bc|$. Another application of Lemma \ref{le:m.8} shows that $z_{i_1i_2} \cdot \bc = 2^{-3/2} \det M(\wt z_{i_1i_2})$ has the form
$2^{-3 /2} (r_1 + r_2 \sqrt{3})$ for some integers $r_1$ and $r_2$ satisfying \eqref{eq:m.8} with $m=2n$. Since $\dist(z_{i_1i_2}, W) >0$, the integers $r_1$ and $r_2$ cannot be equal 0 simultaneously.
If $r_2=0$ then $r_1 \ne 0$ and, therefore,
\begin{align}\label{s24.10}
|z_{i_1i_2}\cdot \bc| = 2^{-3/2} |r_1| \geq 2^{-3/2}.
\end{align}

Next suppose that $r_2 \ne 0$. Then 
\begin{align}\label{s24.5}
z_{i_1i_2}\cdot \bc = 2^{-3/2} (r_1 + r_2 \sqrt{3}) .
\end{align}
We will estimate the last expression using continued fractions. See \cite{Khinchin, RSZ} for the accessible presentation of continued fractions theory.
It is well known that
\begin{align*}
\sqrt{3}=1+\cfrac{1}{1+\cfrac{1}{2+\cfrac{1}{ 1+\cfrac{1}{2+\cfrac{1}{1+\cfrac{1}{2+\cdots}}}}}}
\end{align*}
In the notation of \cite{RSZ}, the continued fraction representing $\sqrt{3}$ can be written as follows,
\begin{align*}
(a_0, a_1, a_2, \dots) = (1,1,2,1,2,1,2,1,2,1,2,\dots).
\end{align*}
Let
\begin{align*}
h_{-2} &= 0, \quad h_{-1} = 1, \quad h_k = a_k h_{k-1} + h_{k-2}
\quad \text{   for  } k\geq 0,\\
g_{-2} &= 1, \quad g_{-1} = 0, \quad g_k = a_k g_{k-1} + g_{k-2}
\quad \text{   for  } k\geq 0.
\end{align*}
It follows from \cite[(2), p.~19]{RSZ} that for every $k$,
\begin{align}\label{s24.4}
 \left| \sqrt{3} - \frac{h_k}{g_k} \right|
>\frac{1}{g_k (g_{k+1} + g_k)}.
\end{align}
The quantity $h_k/g_k$ is a convergent of the continued fraction. Every convergent is nearer to $\sqrt{3}$ than any other fraction whose denominator is less than that of the convergent.
We have $g_{k-1} \geq g_{k-2}$ so $g_k \leq (a_k+1) g_{k-1}$. Hence, for $k\geq 1$,
\begin{align}\label{s24.3}
g_k \leq 3 g_{k-1}.
\end{align}

Recall that $r_2$ satisfies \eqref{eq:m.8} with $m=2n$. 
In  view of \eqref{s24.3}, we can find $k$ such that 
$g_k/3 \leq g_{k-1} \leq |r_2| \leq g_k $. Hence
\begin{align*}
g_{k+1} \leq 9 g_{k-1} \leq 9 |r_2| \leq 9\cdot 4^{2n}/\sqrt{3},
\end{align*}
and, in view of \eqref{s24.4},
\begin{align*}
\left| r_1+r_2\sqrt{3} \right| = 
|r_2| \cdot \left|\frac {r_1}{r_2} + \sqrt{3}\right| &\geq 
|r_2| \cdot \left| \sqrt{3} - \frac{h_k}{g_k} \right|
>\frac{g_k/3}{g_k (g_{k+1} + g_k)}
> \frac{1}{6g_{k+1} } \geq \frac{ \sqrt{3} }{ 54 \cdot 4^{2n}  } .
\end{align*}
We now use \eqref{s24.5} to see that
\begin{align}\label{s24.11}
|z_{i_1i_2}\cdot \bc| = 2^{-3/2}  \left| r_1 + r_2\sqrt{3} \right|
\geq  2^{-3/2}\cdot \frac{ \sqrt{3} }{ 54 \cdot 4^{2n}  }  \geq \sqrt{\frac{3}{8}} \cdot \frac{4^{-2n}}{54}.
\end{align}
In view of \eqref{s24.10}, this estimate holds also in the case when $r_2=0$. 
The  combination of   \eqref{s24.2}, \eqref{s24.1} and \eqref{s24.11} yields
\begin{align*}
\dist(z_{i_1i_2}, V) &\geq  \frac {|z_{i_1i_2}\cdot \bc|}{|\bc|}
\geq \sqrt{\frac{3}{8}} \cdot \frac{4^{-2n}}{54} \cdot \frac{1}{2\cdot 4^{2n} \sqrt{p+1}} 
=  \sqrt{\frac{3}{2}} \cdot \frac{4^{-4n}}{216 \sqrt{p+1}}.
\end{align*}
Since $p=\dim V\leq 2n-1$, the above bound and Definition \ref{s21.1} yield the lemma.
\end{proof}

\section{Eccentricity of spherical convex sets}\label{ecc}

The classical concept of eccentricity applies to ellipses and other conical curves. Roughly speaking, eccentricity measures the elongation of the ellipse; the larger the ratio of the major axis to the minor axis of the ellipse, the larger is its eccentricity. We  use informally the term ``eccentricity of spherical convex sets'' to describe a quantity that is related to, but  is not a direct analogue of, elliptical eccentricity. 
We say that a subset of a sphere is convex if it is the intersection of the sphere with a finite or infinite family of half-spaces whose boundaries pass through the center of the sphere. 
Equivalently, a subset $A$ of a sphere is convex if at least one geodesic between any two points in $A$ is contained in $A$.
Informally,  a convex subset of a sphere has high eccentricity if the ratio of its inradius to its diameter is small.

\begin{remark}
The following  definitions come with some elementary claims whose proofs are left to the reader. 

Consider a family of $n$ pinned balls and an associated graph $G=(\calV_1, \calE_1)$. Recall the definition of $z_{jk}$ from  \eqref{au20.2}, and the subsequent definition of the half space $H_{jk}$. Let $H^{\cap G} = \bigcap _{(j,k)\in \calE_1} \prt H_{jk}$ and let $H^G$ be the orthogonal complement of $H^{\cap G}$. 
The space $H^G$ is spanned by $\{z_{jk},(j,k)\in \calE_1\}$.
Since $0 \in \prt H_{jk}$ for all $(j,k)\in \calE_1$, we have $0\in H^{\cap G}$.
Let
$H^G_* = H^G \cap \bigcap _{(j,k)\in \calE_1}  H_{jk}$.
Every vector $v\in \R^{nd}$ can be uniquely represented as $v= v^{\cap G} + v^G$, with $v^{\cap G }\in H^{\cap G}$ and $v^G \in H^G$. We have
\begin{align}\label{au24.1}
(\calT_{jk}(v))^{\cap G}= v^{\cap G},
\qquad (\calT_{jk}(v))^{ G}= \calT_{jk}(v^{ G}),
\qquad |(\calT_{jk}(v))^{ G}|= |v^{ G}|,
\end{align}
for $(j,k) \in \calE_1$. 
\end{remark}

Recall that $S= \prt \calB(0,1)  \subset \R^{nd}$. 
For a linear subspace $V$ and vector $x$, we will write $\Pi_V(x)$ to denote the orthogonal projection of $x$ on $V$.

\begin{lemma}\label{au14.1}
Suppose that  $G=(\calV_1, \calE_1)$ is a graph associated to a family of pinned balls, $(j,k)\in \calE_1$
and $\calE_2 \subset \calE_1 \setminus \{(j,k)\}$. Assume that 
$ H^G\cap  \bigcap_{(i,m)\in\calE_2} \prt H_{im}$
is one-dimensional and
$v \in S \cap H_*^G\cap  \bigcap_{(i,m)\in\calE_2} \prt H_{im}$.
Then $\dist(v , \prt H_{jk}) \geq \alpha(\bF)$.
\end{lemma}

\begin{proof}
 
 Let $V$ be the subspace spanned by $\left\{ z _{im}, (i,m) \in \calE_2\right\}$. Note that its orthogonal complement is $V^\perp = \bigcap_{(i,m)\in\calE_2} \prt H_{im}$. Let $ z^* _{jk} = \Pi_V(z_{jk})$ be the point in $V$ closest to $ z _{jk}$. By definition of $\alpha(\bF)$, $\dist( z _{jk} ,   z^* _{jk}) \geq \alpha(\bF)$. We also have that $z_{jk} - z^*_{jk} \in H^G\cap V^\perp$, and so $z_{jk} - z^*_{jk} $ is parallel to $v$. Let $y$ be the vector in $\prt H_{jk}$ which is  closest to $v$. Then $v-y$ is parallel to $ z _{jk}$. Hence,
the vectors $v,y,  z _{jk}$ and $ z _{jk} -   z^*_{jk}$ lie in a two-dimensional subspace spanned by $v$ and $ z _{jk}$.  Since $y\in\partial H_{jk}$, the vectors $y$ and $z_{jk}$ are orthogonal. Also, $z^*_{jk}\in V$, and $v\in V^\perp$, so these vectors are orthogonal. Since $|v|=|z_{ij}|=1$, it follows that the triangle with vertices $0,z^*_{jk},z_{jk}$ and the one with vertices $0,y,v$ are congruent and so $|v-y| = |z^*_{jk}-z_{jk}|$.
We conclude that 
\begin{align*}
\dist(v , \prt H_{jk}) =\dist (v, y)= \dist( z _{jk} ,   z^*_{jk}) \geq \alpha(\bF).
\end{align*}
\end{proof}

\begin{lemma}\label{au19.2}
Suppose that $G_1=(\calV_1, \calE_1)$ is a graph associated to a family $\bF$ of pinned balls. 
If $v \in S \cap H^{G_1}_*$ then $\dist(v , \prt H_{jk}) \geq   \alpha(\bF)/(nd)$  for some $(j,k) \in \calE_1$.

\end{lemma}

\begin{proof}

Let $N$ be the dimension of $H^{G_1}$. Since $H^{G_1}\subset \R^{nd}$, $N\leq nd$. Consider any subset $\calE_2 = \{ (i_1,j_1),\ldots,(i_N,j_N)\}$ of $\calE_1$, such that the vectors $\{z_{i_k,j_k} : k =1,\ldots,N\}$ are linearly independent. For $k=1,\ldots,N$  we will denote by $E_k$  the subspace spanned by the vectors $\{z_{i_mj_m} : m\neq k\}$, and will set $\wt w_k = z_{i_k,j_k} - \Pi_{E_k} z_{i_k,j_k} $. Note that $\wt w_k\neq 0$ because $z_{i_kj_k}\notin E_k$, so $w_k := \wt w_k /|\wt w_k|$ is well defined.

Let $G_2=(\calV_2, \calE_2)$, where $\calV_2$ is the set of all vertices in $\calV_1$ which are connected to edges in $\calE_2$.
In view of the definition of $\calE_2$, $H^{G_2} = H^{G_1}$ and, therefore, $H^{G_1}_* \subset H^{G_2}_*$.

The vectors $w_1, w_2, \dots, w_N$ correspond to the the ``vertices'' of the ``convex subset''  $S \cap H_*^{G_2} $ of the sphere $S$. More precisely, each $w_r$ is the only point in the set  $S \cap H^{G_2}_*\cap  \bigcap_{(j,k)\in\calE_2, (j,k) \ne (i_r,j_r)} \prt H_{jk}$.
 Lemma \ref{au14.1} implies that for every $k$ there exists $\ell$ such that 
 $\dist\left(w_{k}, \prt H_{i_{\ell} j_{\ell}}\right) \geq \alpha(\bF)$.

Let $C$ be the convex hull of $\{w_1, w_2, \dots, w_N\} $, an $N$-dimensional simplex. 
Consider any $w \in S \cap H^{G_1}_* \subset S \cap H^{G_2}_*$, fixed from now on, and note that there exists $c\in(0,1]$ such that $w' := c w \in C$. 
There exist unique $\lambda_1, \dots , \lambda_{N} \in [0,1]$ such that $\sum_{j=1}^{N} \lambda_j =1 $ and
$w' = \sum_{j=1}^{N} \lambda_j w_j$. Fix any $k$ such that $\lambda_{k} \geq 1/N$ and $\ell = \ell(k)$ such that 
 $\dist\left(w_{k}, \prt H_{i_{\ell} j_{\ell}}\right) \geq \alpha(\bF)$.

The subspace $\partial H_{i_{\ell} j_{\ell}}$ has the normal vector $z_{i_{\ell} j_{\ell}}$  so for every $w\in H^{G_1}_*$, and our choice of $\ell$ as described above, we have  $\dist(w,\prt H_{i_\ell j_\ell}) = w\cdot z_{i_\ell j_\ell} =c^{-1} w' \cdot z_{i_\ell j_\ell}$. It follows that
\begin{align*}
\dist(w,\prt H_{i_\ell j_\ell}) & \geq w' \cdot z_{i_\ell j_\ell} \geq \lambda_{k} w_{k}  \cdot z_{i_\ell j_\ell} =  \lambda_{k} \dist(w_{k},\prt H_{i_\ell j_\ell}) \geq \frac{\alpha(\bF)}{N}
\geq \frac {\alpha(\bF)}{ nd}.
\end{align*}
This completes the proof.
\end{proof}

\section{Collision number estimate}\label{pinned}

Our main result is stated at the end of this section as Theorem \ref{au21.6}. Its proof is preceded by two lemmas.

\begin{lemma}\label{au19.1}
Assume that the full graph associated with $\bF$ is connected
and consider a graph
 $G=(\calV_1, \calE_1)$ associated to $\bF$. 
If $\delta>0$, $u\in S\cap H^G$ and $u\cdot  z _{jk} \geq -\delta$ for all  $(j,k)\in \calE_1$
then $\dist(u, S \cap H^G_*) \leq 2^{7/2} \delta n (n-1)^2$.

\end{lemma}

\begin{proof}

Let $ w  = (w_1, w_2, \dots, w_n)$ where $w_k = c (x_k-x_1)$ for $k=1,\dots,n$ and $c>0$ is chosen  so that $\sum_{k=1}^n |w_k|^2 = 1$, i.e., $w\in S$.
Recall that we have assumed that the full graph associated with the family of pinned balls is connected. This implies that $|x_k-x_1| \leq 2 (n-1)$ for every $k$, and, therefore,  $\sum_{k=1}^n |x_k-x_1|^2 \leq 4 n (n-1)^2 $. Hence, $c\geq 1/(4n(n-1)^2)$.

Recall notation introduced before Lemma \ref{au14.1}. Let $v= w^G$. Since $z_{ij} \in H^G$ for all $(i,j) \in \calE_1$, we have $v  \cdot  z _{ij}=w  \cdot  z _{ij}$ for all  $(i,j)\in \calE_1$, and, therefore,
\begin{align}\label{au6.1}
 v  \cdot  z _{ij} &=w  \cdot  z _{ij} = 
(w_i - w_j) \cdot (x_i - x_j)
/ |\wt  z  _{ij}|
= 2^{-3/2}
(w_i - w_j) \cdot (x_i - x_j)\\
&= 2^{-3/2}c(x_i-x_1- x_j+x_1) \cdot (x_i - x_j)
 =2^{-3/2} c |x_i - x_j|^2 \nonumber \\
&\geq 2^{-3/2}/(n(n-1)^2).\nonumber
\end{align}
This shows, in particular, that $ v  \in H^G_*$.

Suppose that $u\in S\cap H^G$ and $u\cdot  z _{jk} \geq -\delta$. If $u \in H^G_*$
then $\dist(u,S\cap H^G_*) =0 $ and we are done.

Suppose that $u\notin H^G_*$ and let $K$ be the line segment with the endpoints $u$ and $ v $. This line segment must intersect $\prt H^G_*$. Let $y_1$ be a point in $K \cap \prt H^G_*$. Let $(j,k)$ be such that $y_1 \in \prt H_{jk}$. Let $y_2$ and $y_3$ be the projections of $u$ and $ v $ onto $\prt H_{jk}$, resp. Then all five points $u,  v , y_1, y_2$ and $y_3$ lie in a 2-dimensional plane. The triangles $\Delta(u, y_1, y_2)$ and $\Delta( v , y_1, y_3)$ are similar so
\begin{align}\label{au17.1}
\dist(u, y_1)
= \dist(u, y_2) \frac{\dist( v , y_1)}{\dist( v , y_3)}.
\end{align}
Since $u,  v  \in \calB(0,1)$, $\dist(u,  v ) \leq 2$ and, therefore, 
$\dist(y_1,  v ) \leq 2$. 
The assumption that $u\cdot  z _{jk} \geq -\delta$ and the fact that $u$ and $v$ are on two different sides of $\prt H_{jk}$ imply that 
$\dist(u, y_2) \leq \delta$.
It follows from \eqref{au6.1} that  $\dist( v , y_3) \geq 2^{-3/2}/(n(n-1)^2)$.
These remarks and \eqref{au17.1} imply that
\begin{align*}
 \dist(u, y_1)
= \dist(u, y_2) \frac{\dist( v , y_1)}{\dist( v , y_3)}
\leq \frac{2\delta}{2^{-3/2}/(n(n-1)^2)}
= 2^{5/2} \delta n (n-1)^2.
\end{align*}
Since $u\in S$, the last estimate shows that $\dist( y_1, S) \leq 2^{5/2} \delta n (n-1)^2$. Hence, if we let $y_4 = y_1/ |y_1|$ then $\dist( y_1, y_4) \leq 2^{5/2} \delta n (n-1)^2$. Therefore,
\begin{align*}
\dist(u, S\cap H^G_*) \leq \dist(u, y_4) 
\leq \dist(u, y_1) + \dist( y_1, y_4)
\leq 2^{7/2} \delta n (n-1)^2.
\end{align*}
\end{proof}

Recall the definition of $L(m,\bF)$ from Section \ref{prelim}.

\begin{lemma}
\label{au17.2}
If the full graph associated with $\bF$ is connected
then
\begin{align*}
L(m,\bF)\leq \left( \frac{ 2^{21/2} d n^5 }{\alpha(\bF)}\right)^{m-1}.
\end{align*}
\end{lemma}

\begin{proof}
The proof is by induction on $m$. Consider a graph $G=(\calV_1, \calE_1)$ associated with $\bF$, such that $|\calE_1| = m$. Consider a sequence $\Gamma=(\gamma_t, t \geq 1)$, such that $\gamma_t \in \calE_1$ for every $t$. Let $v(0)=v\in\R^{nd}$, and $v(t)=\calT_{\gamma_j}(v(t-1))$ for $t\geq 1$. 

We will assume that $v\in H^G$. We can make this assumption because \eqref{au24.1} implies that the number of (pseudo-)collisions remains the same when $v$ is replaced by $v^G$. If $v\in H^G$ then $v(t) \in H^G$ for all $t \geq 1$ (once again, because of \eqref{au24.1}) and we will be able to apply Lemma \ref{au19.1}.

If $m=1$ then  $\calE_1=\{(j,k)\}$. The balls $j$ and $k$  can collide at most once. Hence $L(1,\bF)=1$.

Assume that the lemma holds for $m-1 \geq 1$. Let
\begin{align}\label{au21.4}
\delta= \frac{\alpha(\bF)/(2nd)}{2^{7/2}  n (n-1)^2}.
\end{align}
We will define inductively a finite sequence of ``times'' (integers) $s_0, s_1, \ldots,s_\ell $, such that $s_0=0$ and $s_{k+1} \geq s_k+1$. We will denote its length by $\ell+1$, i.e., its last element will be $s_\ell$. 
Suppose that $s_k$ has been defined and $s_k$ has not been declared to be $s_\ell$ so far. 

If no balls collide in $(s_k,\infty)$ (i.e., $v(t+1) = v(t) $ for all $t\geq s_k$) then we let $s_\ell = s_k$.

If some collisions occur in $(s_k,\infty)$ and
\begin{align}\label{oc2.9}
(v_i(s_k) - v_j(s_k))  \cdot (x_i - x_j) \geq -\delta
\end{align}
for all $(i,j)\in \calE_1$ such that balls $i$ and $j$ collide in $(s_k,\infty)$, then we let $s_\ell = s_k$. In the opposite case we let $(i_k,j_k)$ be the first pair in the lexicographic order such that \eqref{oc2.9}
fails and balls $i_k$ and $j_k$ collide in $(s_k,\infty)$. We let $s_{k+1}$ be the first time in $(s_k, \infty)$ when balls $i_k$ and $j_k$ collide. Note that $\ell$ must be finite because collisions can be identified with foldings (see \eqref{au21.2}) and the total number of foldings is finite by  Theorem \ref{au20.1}. The assumption in Theorem \ref{au20.1} saying that the intersection of half-spaces has a non-empty interior is satisfied due to \eqref{au6.1}.

Consider the condition
\begin{align}\label{au17.3}
(v_{i_k}(s_{k+1}-1) - v_{j_k}(s_{k+1}-1))  \cdot (x_{i_k} - x_{j_k}) \leq -\delta/2.
\end{align}
We will consider two cases. First assume that \eqref{au17.3} holds.
In this case, in view of  Lemma \ref{au21.3} (iii) and (iv), we  obtain 
\begin{align}\label{au18.1}
F(s_{k+1}) - F(s_{k}) &\geq F(s_{k+1}) - F(s_{k+1}-1) \\
&= 2 n(v_{j_k}(s_{k+1}-1) - v_{i_k}(s_{k+1}-1))  \cdot (x_{i_k} - x_{j_k})
\geq n \delta. \nonumber
\end{align}

Next suppose that \eqref{au17.3} does not hold.
We will prove that this assumption implies that
\begin{align}
\label{au17.5}
 |v_{i_k}(s_{k+1}-1) - v_{i_k}(s_{k})| + |v_{j_k}(s_{k+1}-1) - v_{j_k}(s_{k})|
 > \delta/8.
\end{align}
Recall that we are assuming that \eqref{oc2.9} does not hold for $i_k$ and $j_k$ and \eqref{au17.3} is false as well. If, in addition, \eqref{au17.5} were not true then we would  have
\begin{align*}
\delta/2 &\leq
|(v_{i_k}(s_k) - v_{j_k}(s_k))  \cdot (x_{i_k} - x_{j_k}) -
(v_{i_k}(s_{k+1}-1) - v_{j_k}(s_{k+1}-1))  \cdot (x_{i_k} - x_{j_k})|\\
&= 
|(-(v_{i_k}(s_{k+1}-1) - v_{i_k}(s_{k})) + (v_{j_k}(s_{k+1}-1) - v_{j_k}(s_{k}) ))\cdot (x_{i_k} - x_{j_k})|
\\
&\leq
(|v_{i_k}(s_{k+1}-1) - v_{i_k}(s_{k})| +
|v_{j_k}(s_{k+1}-1) - v_{j_k}(s_{k})| ) \cdot 2
\leq \delta/4.
\end{align*}
This contradiction proves \eqref{au17.5}.

Let $U$ be the set of times in $(s_k,s_{k+1})$ when balls $i_k$ and $j_k$ collide with any other balls, i.e, $t\in U$ if and only if $v_{i_k}(t) \ne v_{i_k}(t-1)$
or $v_{j_k}(t) \ne v_{j_k}(t-1)$. The first inequality below holds because $F(t)$ is non-decreasing (see Lemma \ref{au21.3} (iv)), the first equality holds by Lemma \ref{au21.3} (ii),  the second inequality by the triangle inequality,  the second equality holds because velocities do not change between collisions, and the last inequality follows from \eqref{au17.5}:
\begin{align*}
F(s_{k+1}) - F(s_{k})  &\geq \sum_{t\in U} F(t) - F(t-1)  \\
&= 4n \sum_{t\in U}  |v_{i_k}(t) - v_{i_k}(t-1)| 
+ 4n \sum_{t\in U}  |v_{j_k}(t) - v_{j_k}(t-1)| 
\\
&\geq 4n \left| \sum_{t\in U}  v_{i_k}(t) - v_{i_k}(t-1) \right| 
+4n \left| \sum_{t\in U}  v_{j_k}(t) - v_{j_k}(t-1) \right| \\
&= 4n |v_{i_k}(s_{k+1}-1) - v_{i_k}(s_{k})|
+ 4n |v_{j_k}(s_{k+1}-1) - v_{j_k}(s_{k})| \geq n\delta /2.
\end{align*}
This and \eqref{au18.1} show that we always have
\begin{align}\label{au18.2}
F(s_{k+1}) - F(s_{k}) \geq n\delta /2.
\end{align}

Recall from \eqref{au21.1} that we have normalized the energy so that $\sum_{i=1}^n |v_i|^2 = 1$. By the conservation of energy, $\sum_{i=1}^n |v_i(t)|^2 = 1$ for all $t\geq 0$. Hence $|v_i(t)|\leq 1$ for all $1\leq i \leq n$ and $t\geq 0$. It follows that, for every $t \geq 0$,
\begin{align*}
F(t) =
\sum_{i,j=1}^n  (v_j(t) - v_i(t)) \cdot (x_j - x_i)
\geq \sum_{i,j=1}^n (-2 \cdot 2) = - 4 n^2.
\end{align*}
Similarly, $F(t) \leq 4 n^2$ for all $t\geq 0$, so $F(t) - F(u) \leq 8 n^2$ for all $0\leq u\leq t$. This, the fact that $F$ is non-decreasing,  and \eqref{au18.2} imply that
\begin{align*}
8 n^2 \geq
F(s_\ell) - F(s_0) = 
\sum_{k=1}^\ell
F(s_{k}) - F(s_{k-1}) \geq \sum_{k=1}^\ell n\delta /2
= \ell n\delta /2,
\end{align*}
and, therefore, in view of \eqref{au21.4},
\begin{align}\label{au21.5}
\ell \leq 16 n /\delta = \frac{16 n \cdot 2^{7/2}  n (n-1)^2}{\alpha(\bF)/(2nd)}
\leq \frac{ 2^{17/2} d n^5 }{\alpha(\bF)}.
\end{align}

On each interval $(s_k, s_{k+1})$ at most $m-1$ pairs of balls have collisions so the maximum number of collisions on such an interval is $L(m-1,\bF)$, by induction assumption. Adding collisions occurring at times $s_k$, we obtain the following upper bound on the number of collisions on $[s_0, s_\ell]$:
\begin{align}\label{au19.4}
\ell L(m-1,\bF) + \ell + 1 = \ell(L(m-1,\bF)+1) +1.
\end{align}

At time $s_\ell$, condition \eqref{oc2.9} holds. By Lemma \ref{au19.1} and \eqref{au21.4}, 
\begin{align*}
\dist( v (s_\ell),S\cap H^G_*) \leq 2^{7/2} \delta n (n-1)^2
= \alpha(\bF)/(2nd).
\end{align*}
Let $ w \in S\cap H^G_*$ be such that
$\dist ( v (s_\ell),  w) \leq \alpha(\bF)/(2nd)$.
Since foldings do not increase distance between points, we obtain for all $k\geq s_\ell$,
\begin{align}\label{au19.3}
\dist ( v (k), w) \leq  \alpha(\bF)/(2nd).
\end{align}
According to Lemma \ref{au19.2}, we have
$\dist( w , \prt H_{i_*j_*}) \geq   \alpha(\bF)/(nd)$
for some $(i_*, j_*)\in \calE_1$.
This and \eqref{au19.3} imply that $ v (k) \in H_{i_*j_*}$ for all $k \geq s_\ell$. Hence, balls $i_*$ and $j_*$ do not collide after time $s_\ell$. This implies that the number of collisions after time $s_\ell$ does not exceed $L(m-1,\bF)$, by the induction assumption. We combine this with \eqref{au19.4} to conclude that 
\begin{align*}
L(m,\bF) &\leq \ell(L(m-1,\bF)+1) +1 + L(m-1,\bF) = (\ell+1)(L(m-1,\bF)+1)\\
&\leq 4 \ell L(m-1,\bF).
\end{align*}
Hence, by \eqref{au21.5},
\begin{align*}
L(m,\bF) \leq L(1,\bF) (4 \ell)^{m-1}
\leq 1 \cdot \left(4 \cdot \frac{ 2^{17/2} d n^5 }{\alpha(\bF)}\right)^{m-1}
=
\left( \frac{ 2^{21/2} d n^5 }{\alpha(\bF)}\right)^{m-1}.
\end{align*}
This completes the proof.
\end{proof}

Recall that $\tau_d$ denotes
the  kissing number  of a $d$-dimensional ball, i.e.,  the maximum
number of mutually nonoverlapping translates of the ball that can be arranged
so that they all touch the ball.

\begin{theorem}\label{au21.6}
The  number of (pseudo-)collisions of $n$ pinned balls in $\R^d$, i.e., the  number of distinct vectors in the sequence $(v(t), t\geq 0)$, is bounded above by
\begin{align}\label{s28.2}
\left( \frac{ 2^{21/2} d n^5 }{\alpha(\bF)}\right)^{\tau_d n/2-1}.
\end{align}
\end{theorem}

\begin{proof}
First suppose that the full graph associated with $\bF$ is connected.
If the number of balls is $n$ then the number  of pairs of touching balls is bounded by $\tau_d n /2$. The claim now follows from Lemma \ref{au17.2}.

Suppose that the full graph associated with $\bF$ is not connected. Then it consists of $k$ connected components for some $k\geq 2$. Let $n_j$ be the number of vertices in the $j$-th connected component. Then $n_1 + \dots + n_k = n$. The balls in different connected components do not interact so we can use the bound in \eqref{s28.2} for the number of collisions within each connected component. Let 
\begin{align*}
f(n) = \left( \frac{ 2^{21/2} d n^5 }{\alpha(\bF)}\right)^{\tau_d n/2-1}
\end{align*}
and note that
$f(n) \geq f(n_1) + \dots + f(n_k)$ if $n_1 + \dots + n_k = n$.
\end{proof}

The following corollary follows easily from Theorem \ref{au21.6} and Lemmas \ref{m3.1} and \ref{s26.10}.

\begin{corollary}\label{m5.1}
(i) If the full graph associated with the family of $n$ pinned balls is a tree then the number of (pseudo-)collisions is not greater than
\begin{align}\label{s28.3}
\left( 2^{17/2} d n^6 \right)^{\tau_d n/2-1}.
\end{align}

(ii) 
Suppose that $d=2$ and the centers of pinned balls $\bF$ belong to the triangular lattice $\calX$ (see Definition \ref{s27.1}). Then the number of (pseudo-)collisions is not greater than
\begin{align}\label{s28.4}
\left( 2^{21/2} \cdot 2 n^5  \frac{432}{\sqrt{3}} \cdot 4^{4n}\sqrt{n}\right)^{\tau_2 n/2-1}
< \left( 10^6 n^{11/2}  4^{4n}\right)^{3 n-1}.
\end{align}

\begin{remark}

(i) There are both intuitive and theoretical 
reasons to think that pinned balls model is closely related to the stage in the moving balls evolution that generates the largest number of collisions. For instance, for appropriate $c,\eps>0$, Theorem 1.3 in \cite{BuD18b} asserts that if a total of more than $n^{cn}$ collisions occur, then at least $n^{\eps n}$ of them will have to occur in an interval of time during which a subset of the balls form a very tight configuration. 

For a general configuration of pinned balls, our estimate \eqref{s28.2} on the maximum number of collisions is weaker than the estimates \eqref{s26.1} and \eqref{s26.2} for the moving balls because our estimate depends on $\alpha(\bF)$ and hence on the positions of the balls.

On the other hand, our estimates \eqref{s28.3} and \eqref{s28.4} for special pinned ball configurations are better than those in \eqref{s26.1} and \eqref{s26.2}, for a fixed dimension $d$ and the number of balls $n$ going to infinity.

(ii) Recall from Remark \ref{s28.6} that there is no universal bound (depending only on the number of different half-spaces) for the
size of the orbit of a point acted upon by a sequence of foldings. In view of the fact that the bounds in \eqref{s26.1} and \eqref{s26.2} do not depend on the initial conditions (positions and velocities of moving balls), it is conceivable that there might be an upper bound for the maximum number of collisions of pinned balls depending only on the number $n$ of balls and the dimension $d$. The question of existence of such a bound is left as an open problem.

(iii) Part (ii) of the remark suggests the following open problem.  Under what conditions on a system of $n$ pinned balls and sequence $\Gamma$ (see Section \ref{prelim} for the definition) is it true that for every $\eps >0$ there exists a family of $n$ elastically colliding moving balls such that their total energy is equal to 1, the center of the $k$-th  moving ball stays within $\eps$ from the center of the $k$-th pinned ball over the time interval $[0,\eps]$, and the sequence of collisions of the moving balls over the time interval $[0,\eps]$ is $\Gamma$, i.e., the  pairs of moving balls collide in the same order as the corresponding pairs of pinned balls? 

(iv) A lower bound for the number of collisions of $n$ pinned balls is $n^3/27$ for $n\geq 3$ and $d\geq 2$. This is the same bound as the one for a family of moving balls, presented in \cite{BD}.
We will not give a formal proof for this lower bound for the system of pinned balls because the description of the main example and the arguments given in \cite{BD} for moving balls clearly show that the same bound applies to pinned balls. 

\end{remark} 

\end{corollary}

\section{Acknowledgments}
We are grateful to  Branko Gr\"unbaum, Jaime San Martin and Rekha Thomas for very helpful advice.

\bibliographystyle{alpha}
\bibliography{hard}

\end{document}